\documentclass[11pt]{amsart}
\usepackage{lmodern}
\usepackage{amsmath, amsthm, amssymb, amsfonts}
\usepackage[normalem]{ulem}
\usepackage{hyperref}

\usepackage{mathrsfs}

\usepackage{verbatim} 
\usepackage{longtable}

\usepackage{mathtools}

\usepackage{tikz}
\usetikzlibrary{decorations.pathmorphing}
\tikzset{snake it/.style={decorate, decoration=snake}}

\usepackage{caption}

\usepackage{tikz-cd}
\usetikzlibrary{arrows}

\theoremstyle{plain}
\newtheorem{thm}{Theorem}[section]

\newtheorem{lem}[thm]{Lemma}
\newtheorem{prop}[thm]{Proposition}

\newtheorem{question}[thm]{Question}

\theoremstyle{definition}

\theoremstyle{remark}
\newtheorem{rmk}[thm]{Remark}

\newcommand{\BC}{{\mathbb{C}}}

\newcommand{\BN}{{\mathbb{N}}}

\newcommand{\BP}{{\mathbb{P}}}
\newcommand{\BQ}{{\mathbb{Q}}}

\newcommand{\BZ}{{\mathbb{Z}}}

\newcommand{\CC}{{\mathcal C}}

\newcommand{\CF}{{\mathcal F}}

\newcommand{\CO}{{\mathcal O}}

\DeclareFontFamily{OT1}{rsfs}{}
\DeclareFontShape{OT1}{rsfs}{n}{it}{<-> rsfs10}{}
\DeclareMathAlphabet{\curly}{OT1}{rsfs}{n}{it}

\newcommand{\git}{\mathbin{
  \mathchoice{/\mkern-6mu/}
    {/\mkern-6mu/}
    {/\mkern-5mu/}
    {/\mkern-5mu/}}}

\usepackage{tikz}
\usepackage{lmodern}
\usetikzlibrary{decorations.pathmorphing}

\begin{document}
\title[On intersection cohomology and Lagrangian fibrations]{On intersection cohomology and Lagrangian fibrations of irreducible symplectic varieties}
\date{\today}

\author[C. Felisetti]{Camilla Felisetti}
\address{University of Trento}
\email{camilla.felisetti@unitn.it}

\author[J. Shen]{Junliang Shen}
\address{Yale University}
\email{junliang.shen@yale.edu}

\author[Q. Yin]{Qizheng Yin}
\address{Peking University}
\email{qizheng@math.pku.edu.cn}

\begin{abstract}
We prove several results concerning the intersection cohomology and the perverse filtration associated with a Lagrangian fibration of an irreducible symplectic variety. We first show that the perverse numbers only depend on the deformation equivalence class of the ambient variety. Then we compute the border of the perverse diamond, which further yields a complete description of the intersection cohomology of the Lagrangian base and the invariant cohomology classes of the fibers. Lastly, we identify the perverse and Hodge numbers of intersection cohomology when the irreducible symplectic variety admits a symplectic resolution. These results generalize some earlier work by the second and third authors in the nonsingular case.
\end{abstract}

\maketitle

\setcounter{tocdepth}{1} 

\setcounter{section}{-1}

\section{Introduction}

We work over the complex numbers $\BC$.

\subsection{Irreducible symplectic varieties}
An irreducible symplectic manifold is a simply connected K\"ahler manifold $M$ with $H^0(M, \Omega_M^2)$ spanned by a nowhere degenerate holomorphic~$2$-form. Irreducible symplectic manifolds can be viewed as higher dimensional analogues of $K3$ surfaces, whose geometry and topology have been studied intensively for decades from different angles. By the Beauville--Bogomolov decomposition, these manifolds form one of the three types of manifolds which are building blocks of compact K\"ahler manifolds with numerically trivial canonical bundles.

The purpose of this article is to study a class of algebraic varieties, called \emph{irreducible symplectic varieties}, which are close to irreducible symplectic manifolds but are allowed to be singular. 

Following \cite[Definition 8.16]{GKP}, we say that $M$ is irreducible symplectic, if $M$ is a normal projective variety with trivial canonical divisor and at worst canonical singularities, such that:
\begin{enumerate}
    \item[(i)] there is a reflexive $2$-form $\sigma$ which is non-degenerate on the regular part $M_{\mathrm{reg}} \subset M$, and
    \item[(ii)] for every finite quasi-\'etale\footnote{This means that the morphism is \'etale in codimension one.} morphism $f: M' \to M$,  the exterior algebra of reflexive forms on $M'$ is generated by the reflexive pullback $f^*\sigma$ of the symplectic form.
\end{enumerate}
Nonsingular irreducible symplectic varieties recover projective irreducible symplectic manifolds \cite[Remark 8.19]{GKP}. Moreover, analogous to the Beauville--Bogomolov decomposition in the nonsingular setting, irreducible symplectic varieties form one of the three building blocks of projective varieties with klt singularities and numerically trivial canonical divisors \cite{GKP, DG, Gue, Dru, GGK, HP}.

We explore cohomological structures of irreducible symplectic varieties with focus on the interaction with the topology of \emph{Lagrangian fibrations}.

\subsection{Perverse filtrations and Lagrangian fibrations}
Let $\pi: X \rightarrow Y$ be a proper morphism. The perverse $t$-structure on the constructible derived category $D_c^b(Y)$ induces an increasing filtration on the \emph{intersection cohomology} $\mathrm{IH}^*(X, \BC)$,
\begin{equation} \label{Perv_Filtration}
    P_0\mathrm{IH}^\ast(X, \BC) \subset P_1\mathrm{IH}^\ast(X, \BC) \subset \dots \subset P_k\mathrm{IH}^\ast(X, \BC) \subset \dots \subset \mathrm{IH}^\ast(X, \BC),
\end{equation}
called the \emph{perverse filtration} associated with $\pi$. Perverse filtrations play important roles in the study of Hitchin systems \cite{dCHM1, dCMS}, irreducible symplectic manifolds \cite{SY, HLSY}, and enumerative geometry \cite{MT, MS}. We refer to Section \ref{Sec1} for a brief review of the subject.

The filtration (\ref{Perv_Filtration}) is governed by the topology of the map $\pi: X\rightarrow Y$. Some important invariants are the \emph{perverse numbers}
\[
^\mathfrak{p}\mathrm{Ih}^{i,j}(\pi)= \dim \mathrm{Gr}^P_i \mathrm{IH}^{i+j}(X, \BC)= \dim \left( P_i \mathrm{IH}^{i+j}(X, \BC)/ P_{i-1}  \mathrm{IH}^{i+j}(X, \BC) \right).
\]

Now we consider Lagrangian fibrations of irreducible symplectic varieties associated with the (reflexive) symplectic form. The following theorem lists a few basic properties of perverse numbers and the topology of Lagrangian fibrations.

\begin{thm}\label{thm1}
Let $M$ and $M'$ be two irreducible symplectic varieties of dimension $2n$ with second Betti numbers $b_2(M)$ and $b_2(M')$ at least $5$. Let $\pi: M \to B$ and $\pi': M' \to B'$ be two Lagrangian fibrations.
\begin{enumerate}
    \item[(a)] If $M$ is deformation equivalent to $M'$, then we have
\[
^\mathfrak{p}\mathrm{Ih}^{i,j}(\pi) = {^\mathfrak{p}\mathrm{Ih}}^{i,j}(\pi').
\]
    \item[(b)] We have \begin{equation}\label{thm2_eq}
{^\mathfrak{p}\mathrm{Ih}^{0,d}(\pi)} = {^\mathfrak{p}\mathrm{Ih}^{d,0}(\pi)} = \begin{cases}
1, & d = 2k;\\
0, & d= 2k+1.
\end{cases}
\end{equation}
\item[(c)] The intersection cohomology of the Lagrangian base $B$ is given by
\[
\mathrm{IH}^d(B, \BC) = 
\begin{cases}
\langle \beta^k \rangle, & d=2k;\\
0, & d=2k+1,
\end{cases}
\]
where $\beta$ is an ample divisor class on $B$; in particular, we have 
\[
\mathrm{IH}^d (B ,\BC) \simeq H^d (\BP^n, \BC);
\]

\item[(d)] The restriction of the intersection cohomology $\mathrm{IH}^d(M, \BC)$ to a nonsingular fiber $M_b$ is given by
\[
\mathrm{Im} \left\{\mathrm{IH}^d(M, \BC) \to H^d(M_b, \BC)\right\} = 
\begin{cases}
\langle \eta^k|_{M_b} \rangle, & d=2k;\\
0, & d=2k+1,
\end{cases}
\]
where $\eta$ is a $\pi$-relative ample divisor class on $M$.\footnote{By \cite[Theorem 3 (3)]{Sch}, a general fiber of a Lagrangian fibration associated with any irreducible symplectic variety is an abelian variety.}

\end{enumerate}

\end{thm}

When $M$ is nonsingular, it is a folklore conjecture that the Lagrangian base is a projective space $B \simeq \BP^n$. This was confirmed by Hwang \cite{Hw} assuming $B$ nonsingular, and remained wide open in general. See also \cite{Ou, HX, BK} for recent progress.

An interesting phenomenon occurs in the singular case that the base $B$ may fail to be a projective space. In \cite[Theorem 1.9]{MaBase}, Matsushita constructed an irreducible symplectic variety\footnote{The fact that Matsushita's example is irreducible symplectic is verified by \cite[Proposition 2.5]{Per}.} whose Lagrangian base is a Fano variety with quotient singularities. Nevertheless, Theorem \ref{thm1} (c) guarantees that $B$ shares the same intersection cohomology as $\BP^n$. Finding more topological constraints for Lagrangian bases of irreducible symplectic varieties remains a intriguing question.

\begin{rmk}
Mirko Mauri has kindly informed us that the intersection cohomology and ordinary cohomology of $B$ actually coincide, or more precisely, that $\mathrm{IC}_B \simeq \mathbb{Q}_B[n]$ in $D^b_c(B)$. Indeed, by \cite[Theorem 3 (4)]{Sch}, a Lagrangian fibration $\pi: M \to B$ has equidimensional fibers and no fiber is contained in the singular locus of $M$. Hence the argument of \cite[Proposition~1.10 (ii)]{HM} (stated for nonsingular $M$) extends to the singular setting. As noticed in \cite[Remark 1.11]{HM}, it is expected that~$B$ should have finite quotient singularities.
\end{rmk}

Theorem \ref{thm1} was proven when $M$ is nonsingular in \cite{SY,HLSY}. Moreover, in the nonsingular setting, a stronger version of (a) was deduced, namely that perverse numbers are identified with Hodge numbers of the ambient variety $M$. We refer to \cite[Section 0.4.2]{SY} for connections to enumerative geometry of $K3$ surfaces. In the following section, we discuss this ``perverse = Hodge" phenomenon in the singular case.



\subsection{Perverse = Hodge}

For a possibly singular irreducible symplectic variety $M$ with a Lagrangian fibration $\pi: M \to B$, the intersection cohomology $\mathrm{IH}^*(M, \BC)$ carries a pure Hodge structure so that we may consider the associated Hodge numbers
\[
\mathrm{Ih}^{i,j}(M) = h^{i,j}\left( \mathrm{IH}^*(M, \BC)\right) \in \BZ, \quad \textrm{for all }\, i,j \in \BN.
\]
It is a natural question if the ``perverse = Hodge" identity still holds in general.

\begin{question}\label{Q_P=F}
     Is it true that 
    \begin{equation}\label{P=F}
        {^\mathfrak{p}\mathrm{Ih}}^{i,j}(\pi) = \mathrm{Ih}^{i,j}(M)\,?
    \end{equation}
\end{question}

The following result, which generalizes \cite[Theorem 0.2]{SY}, provides an affirmative answer to Question \ref{Q_P=F} when $M$ admits a \emph{symplectic resolution}\footnote{A resolution $f: M' \to M$ is symplectic if $f^*\sigma$ is non-degenerate on $M'$.} (with no restriction on the second Betti number).

\begin{thm}\label{thm0.5}
Assume that $M$ is an irreducible symplectic variety admitting a symplectic resolution and a Lagrangian fibration $\pi: M \to B$, then (\ref{P=F}) holds for $M$.
\end{thm}

As a key step in the proof of Theorem \ref{thm0.5}, we compute the \emph{Looijenga--Lunts--Verbitsky (LLV) Lie algebra} \cite{LL, Ver95}, which is the structure Lie algebra of the intersection cohomology generated by all Lefschetz $\mathfrak{sl}_2$-triples, of an irreducible symplectic variety admitting a symplectic resolution. This generalizes \cite{Ver95, Ver96} and may be of independent interest. 

\begin{thm} \label{thm0.4}
Let $M$ be an irreducible symplectic variety admitting a symplectic resolution. Then the LLV algebra associated with the intersection cohomology of $M$ is naturally isomorphic to $\mathfrak{so}(b_2(M) +2)$.
\end{thm}

\subsection{Connection to the moduli of Higgs bundles}
The study of perverse filtrations of Lagrangian fibrations associated with singular irreducible symplectic varieties is partly motivated by the $P=W$ conjecture for singular Higgs moduli spaces.

Let $C$ be a Riemann surface of genus $g\geq 2$. Non-abelian Hodge theory \cite{Simp, Simp95, HT1} induces a canonical diffeomorphism between the moduli space $M_{\mathrm{n,d}}$ of rank $n$ degree $d$ semistable Higgs bundles on $C$ and the corresponding character variety
\[
M'_{n,d} =  \Big{\{}a_k, b_k \in \mathrm{GL}_n,~k=1,2,\dots,g \,| \,\prod_{j=1}^g [a_j, b_j] = \zeta_n^d \cdot \mathrm{Id}_n \Big{\}}\git \mathrm{GL}_n,\quad \zeta_n:=e^{\frac{2\pi \sqrt{-1} }{n}},
\]
where the $\mathrm{GL}_n$-quotient is with respect to the conjugation action. When $\mathrm{gcd}(n,d)=1$, the moduli spaces $M_{n,d}$ and $M'_{n,d}$ are nonsingular. The $P=W$ conjecture by de Cataldo--Hausel--Migliorini \cite{dCHM1} predicts that the perverse filtration associated with the Hitchin fibration~\mbox{$h_{n,d}: M_{n,d} \to A_{n,d}$} matches the double indexed weight filtration associated with the mixed Hodge structure on $M'_{n,d}$. Moreover, this striking phenomenon is expected to hold more generally without the coprime assumption of $n,d$ if we work with intersection cohomology \cite{dCM, FM, Mauri}:
\begin{equation}\label{P=W}
P_k\mathrm{IH}^i(M_{n,d}, \BC) = W_{2k}\mathrm{IH}^i(M'_{n,d}, \BC).
\end{equation}
Now we consider two integers $d_1, d_2$ with
\begin{equation}\label{eqn123}
\mathrm{gcd}(n,d_1) = \mathrm{gcd}(n,d_2).
\end{equation}
The character varieties $M'_{n,d_1}$ and $M'_{n,d_2}$ are Galois conjugate via an automorphism of $\BQ[\zeta_n]$ sending $\zeta_n^{d_1}$ to $\zeta_n^{d_2}$. The Galois conjugation induces an isomorphism preserving the weight filtrations
\[
 W_k\mathrm{IH}^i(M_{n,d_1} , \BC) \xrightarrow{\simeq} W_k\mathrm{IH}^i(M_{n,d_2},\BC).
\]
Hence the $P=W$ conjecture (\ref{P=W}) predicts that as long as (\ref{eqn123}) holds, we have the following identity concerning the perverse numbers
\begin{equation}\label{galois}
   {^\mathfrak{p}\mathrm{Ih}}^{i,j}(h_{n,d_1}) = {^\mathfrak{p}\mathrm{Ih}}^{i,j}(h_{n,d_2}).
\end{equation}

When $\mathrm{gcd}(n,d)\neq 1$, the Higgs moduli space $M_{n,d}$ is a (non-proper) symplectic variety with the Hitchin fibration Lagrangian. A ``compact analogue" (see \cite{DEL, dCMS}) of the Hitchin fibration associated with $M_{n,d}$ is the Beauville--Mukai system 
\[
\pi^S_{\beta, \chi}: M_{\beta, \chi}(S) \to |\CO_S(\beta)|, \quad \CF \mapsto \mathrm{supp}(\CF).
\]
Here $S$ is a $K3$ surface, $\beta$ is a curve class on $S$ with $\beta^2 = 2g-2$, $|\CO_S(\beta)|$ is the linear system, and $M_{\beta, \chi}(S)$ is the moduli of semistable sheaves $\CF$ on $S$ with respect to a general primitive polarization satisfying that 
\[
[\mathrm{supp}(\CF)] = \beta \in H_2(S, \BZ), \quad \chi(\CF) = \chi.
\]
The moduli space $M_{\beta, \chi}(S)$ is irreducible symplectic \cite{PR}. Furthermore, Theorem \ref{thm1} (a) together with \cite[Theorem 1.17]{PR} implies that
\begin{equation}\label{compact_PW}
    {^\mathfrak{p} \mathrm{Ih}^{i,j}\left(\pi^{S_1}_{\beta_1,\chi_1}\right)} = {^\mathfrak{p}\mathrm{Ih}^{i,j}\left(\pi^{S_2}_{\beta_2,\chi_2}\right)}
\end{equation}
as long as 
\[
\beta_1^2 = \beta_2^2 = 2g-2, \quad \mathrm{gcd}(\mathrm{div}(\beta_1), \chi_1) = \mathrm{gcd}(\mathrm{div}(\beta_2), \chi_2).
\]
Here $\mathrm{div}(-)$ stands for the divisibility of the curve class. The identity (\ref{compact_PW}) can be viewed as a compact analogue of the identity (\ref{galois}) predicted by the $P=W$ conjecture and Galois conjugation.


\subsection{Acknowledgements}
Our original motivation was to rewrite and generalize some results in \cite{SY} to the singular setting using monodromy symmetries. Similar monodromy arguments have appeared in the seminar notes of Huybrechts--Mauri \cite{HM}. We sincerely thank Mirko Mauri for his careful reading of the manuscript and for numerous suggestions. We are also grateful to Davesh Maulik, Giovanni Mongardi, Arvid Perego, and Chuanhao Wei for helpful discussions.

C.~F.~is supported by PRIN 2017 “Moduli Theory and Birational Classification” and GNSAGA. J.~S.~is supported by the NSF grant DMS 2134315. Q.~Y.~is supported by the NSFC grants 11831013 and 11890661.

\section{Perverse filtrations and Proof of Theorem \ref{thm1} (a)}\label{Sec1}

\subsection{Perverse filtrations} We briefly review the perverse filtration associated with a projective map $\pi: X \to Y$. Throughout, we assume that $\pi$ has equidimensional fibers with $\dim X = 2 \dim Y = 2n$ for convenience.  


For $k \in \BZ$, let $^\mathfrak{p}\tau_{\leq k}$ be the truncation functor associated with the perverse $t$-structure. Given an object $\CC \in D_c^b(Y)$ in the truncated derived category of constructible sheaves, there is a natural morphism $^\mathfrak{p}\tau_{\le k}\CC \rightarrow \CC$. For $\pi:X \to Y$, we thus obtain the morphism
\[
^\mathfrak{p}\tau_{\le k}R\pi_\ast \mathrm{IC}_X[-n] \rightarrow R\pi_\ast \mathrm{IC}_X[-n],
\]
which further induces a morphism of (hyper-)cohomology groups
\begin{equation}\label{perv_filt}
H^{d-n}\Big{(}Y, \, ^\mathfrak{p}\tau_{\le k}(R\pi_\ast \mathrm{IC}_X[-{n}]) \Big{)} \rightarrow \mathrm{IH}^d(X, \BC).
\end{equation}
The $k$-th piece of the perverse filtration is defined to be the image of (\ref{perv_filt}). 


Since every fiber of $\pi$ has dimension $n$, the functor $R\pi_*[-n]$ is perverse left $t$-exact and the functor
\[
R\pi_* [n] = R\pi_! [n]
\]
is perverse right $t$-exact. Consequently, the perverse filtration starts with the $0$-th piece and terminates at the $2n$-th piece. In particular, the perverse number ${^\mathfrak{p}\mathrm{Ih}}^{i,j}(\pi)$ is nontrivial only~if 
\begin{equation*}\label{ineq}
0 \leq i \leq 2n, \quad 0 \leq j \leq 2n.
\end{equation*}
Furthermore, the hard Lefschetz theorems for perverse cohomology groups provide the symmetry
\[
{^\mathfrak{p}\mathrm{Ih}}^{i,j}(\pi) = {^\mathfrak{p}\mathrm{Ih}}^{2n-i,j}(\pi) = {^\mathfrak{p}\mathrm{Ih}}^{i,2n-j}(\pi).
\]


\subsection{Perverse filtration for projective bases}
We review another description of the perverse filtration associated with $\pi: X \to Y$ when $Y$ is projective.

We fix $\beta$ to be an ample class on $Y$, and we consider
\[
L = \pi^* \beta \in H^2(X, \BC).
\]
The class $L$ acts on $\mathrm{IH}^*(X, \BC)$ as a nilpotent operator via cup product. The following proposition shows that the filtration (\ref{Perv_Filtration}) is completely described by an ample class on the base.  

\begin{prop}[{\cite[Proposition 5.2.4]{dCM0}}] \label{Prop1.3}
We have
\[
P_k\mathrm{IH}^d(X, \BC) = \sum_{i\geq 1} \left(
\mathrm{Ker}(L^{n+k+i-d}) \cap \mathrm{Im}(L^{i-1}) \right) \cap \mathrm{IH}^d(X, \BC).
\]
\end{prop}

\subsection{Isotropic classes}
Compared to the moduli of Higgs bundles, the perverse filtration of an irreducible symplectic variety $M$ is more manageable due to the fact that, in view of Proposition \ref{Prop1.3}, it can be completely described by an \emph{isotropic} class with respect to the (complexified) \emph{Beauville--Bogomolov--Fujiki (BBF) quadratic form} \cite{Nam,Kir,Sch}
\[
q_M: H^2(M, \BC) \to \BC,
\]
where the monodromy symmetries come into play. This circle of ideas also plays a key role in the recent progress on the $P=W$ conjecture concerning perverse filtrations of Hitchin fibrations \cite{dCMS}.

Let $M$ be an irreducible symplectic variety of dimension $2n$. Any class $\gamma \in H^2(M, \BC)$ acts on the intersection cohomology $\mathrm{IH}^*(M, \BC)$ via cup product as a nilpotent operator. We say that $\gamma \in H^2(M, \BC)$ is isotropic if $q_M(\gamma) = 0$. In view of Proposition \ref{Prop1.3}, we define for any isotropic class $\gamma$ an increasing filtration 
\begin{equation}\label{filtration}
P^\gamma_k\mathrm{IH}^d(M, \BC) = \sum_{i\geq 1} \left(
\mathrm{Ker}(\gamma^{n+k+i-d}) \cap \mathrm{Im}(\gamma^{i-1}) \right) \cap \mathrm{IH}^d(M, \BC).
\end{equation}
In particular, for a Lagrangian fibration $\pi: M \to B$, if $\gamma$ is given by the pullback of an ample class on $B$, then the Fujiki relations \cite[Theorem 2]{Sch} imply that $\gamma$ is an isotropic class. Consequently, (\ref{filtration}) (with different choices of $\gamma$) recovers the perverse filtrations associated with any Lagrangian fibrations $\pi: M \to B$.

\subsection{Proof of Theorem \ref{thm1} (a)}

By the discussion above, Theorem \ref{thm1} (a) is a consequence of the following more general statement concerning the filtrations (\ref{filtration}).

\begin{prop}\label{Prop1.4}
Assume $b_2(M) \geq 5$. Let $\gamma_1, \gamma_2 \in H^2(M, \BC)$ be two nonzero isotropic classes. Then we have
\[
\dim P^{\gamma_1}_k\mathrm{IH}^d(M, \BC) = \dim P^{\gamma_2}_k\mathrm{IH}^d(M, \BC).
\]
\end{prop}

\begin{proof}
Let $\Lambda$ be a lattice isomorphic to the second cohomology $H^2(M, \BZ)$ endowed with the BBF form $q_M$.\footnote{{Notice that $H^2(M, \BZ)$ is torsion-free. In fact, by the universal coefficient formula, the torsion part of~$H^2(M,\BZ)$ comes from $H_1(M,\BZ)$, which is $0$ by \cite[Corollary 13.3]{GGK}.}} We denote by $G_M$ the monodromy group of the $\Lambda$-marked irreducible symplectic varieties deformation equivalent to $M$. By \cite[Theorem~1.1~(1)]{BL}, we have
\begin{equation}\label{subgroup}
G_M \subset O(\Lambda) \subset O(\Lambda_\BC)
\end{equation}
where the first inclusion is given by a subgroup of finite index. We consider the subgroups of~\eqref{subgroup} contained in the connected component $\mathrm{SO}(\Lambda_\BC) \subset O(\Lambda_\BC)$:
\[
G_M^\circ = G_M\cap \mathrm{SO}(\Lambda) \subset \mathrm{SO}(\Lambda) \subset \mathrm{SO}(\Lambda_\BC)
\]
where the first inclusion is a finite index subgroup and the second inclusion is a Zariski dense subset by the Borel density theorem. In particular, we obtain that $G_M^\circ$ is Zariski dense in~$\mathrm{SO}(\Lambda_\BC)$.

Now we consider the action of the monodromy group, which induces the symmetry of dimensions
\begin{equation}\label{sym1}
\dim P^{\gamma_i}_k\mathrm{IH}^d(M, \BC) = \dim P^{g\gamma_i}_k\mathrm{IH}^d(M, \BC), \quad \textrm{for all }\, g\in G^\circ_M, \, i = 1,2.
\end{equation}
Since both $\gamma_1$ and $\gamma_2$ are isotropic classes, they lie in the same orbit $\Omega$ of the natural $\mathrm{SO}(\Lambda_\BC)$ action on $H^2(M, \BC)$. By the definition (\ref{filtration}),  the function on $\Omega$:
\begin{equation}\label{function}
\gamma \in \Omega \mapsto \dim P^\gamma_k\mathrm{IH}^d(M, \BC)
\end{equation}
which is expressed in terms of kernels and images of the operators $\cup \gamma^j$, is constructible on $\Omega$. Moreover, by (\ref{sym1}), it takes constant values on the sets
\[
S_1 = \{g\gamma_1 \in \Omega \, |\,g\in G^\circ_M\},\quad S_2 = \{g\gamma_2 \in \Omega\,|\,g\in G^\circ_M\}.
\]
Since both $S_i$ are Zariski dense in $\Omega$, we conclude that the function (\ref{function}) has to be constant. This completes the proof of the proposition.
\end{proof}


\section{Lagrangian fibrations and Proof of Theorem \ref{thm1} (b,c,d)}

\subsection{Proof of Theorem \ref{thm1} (b,c,d)}

It is explained in \cite[Section 3.2]{SY} that (c,d) are consequences of $(b)$.\footnote{Since we work with possibly singular varieties, the only change for \cite[Section 3.2]{SY} is to replace the trivial local system by the (shifted) intersection cohomology complex $\mathrm{IC}_M[-2n]$, and all the arguments work identically.} More precisely, for a Lagrangian fibration $\pi: M \to B$ with $M_b$ a nonsingular fiber, the argument in \cite[Section 3.2]{SY} actually shows the inequalities
\[
\dim \mathrm{IH}^d(B, \BC) \leq {^\mathfrak{p}\mathrm{Ih}}^{0,d}(\pi), \quad \mathrm{Im} \left\{\mathrm{IH}^d(M, \BC) \to H^d(M_b, \BC)\right\}  \leq {^\mathfrak{p}\mathrm{Ih}}^{d,0}(\pi).
\]
On the other hand, by considering the powers of an ample class on $B$ and a $\pi$-relative ample class on $M$, we clearly have 
\[
\dim \mathrm{IH}^d(B, \BC) \geq
\begin{cases}
 1 , & d=2k;\\
0, & d=2k+1,
\end{cases} 
\]
and
\[
\dim \mathrm{Im} \left\{\mathrm{IH}^d(M, \BC) \to H^d(M_b, \BC)\right\} 
\geq
\begin{cases}
 1 , & d=2k;\\
0, & d=2k+1.
\end{cases}
\]
Hence Theorem \ref{thm1} (b,c,d) are all deduced from the following weaker version of (b).

\begin{prop}\label{prop22}
We have
\[
{^\mathfrak{p}\mathrm{Ih}}^{0,d}(\pi) \leq
\begin{cases}
 1 , & d=2k;\\
0, & d=2k+1,
\end{cases} \quad
{^\mathfrak{p}\mathrm{Ih}}^{d,0}(\pi) \leq
\begin{cases}
 1 , & d=2k;\\
0, & d=2k+1.
\end{cases}
\]
\end{prop}

We prove Proposition \ref{prop22} in Section \ref{Sec2.4} which completes the proof of Theorem \ref{thm1}.  \qed

\subsection{Reflexive symplectic forms and filtrations}
Let $M$ be an irreducible symplectic variety with a reflexive symplectic form $\sigma$. We denote by $\Omega_M^{[d]}$ the sheaf of reflexive $d$-forms, \emph{i.e.},
\[
\Omega_M^{[d]} = j_*\Omega^d_{M_{\mathrm{reg}}},
\]
where $j: M_{\mathrm{reg}} \hookrightarrow M$ is the open embedding of the regular part. We first bound the border of the Hodge diamond of the intersection cohomology of $M$.

\begin{lem}\label{lem2.2}
We have
\[
\mathrm{Ih}^{0,d}(M) \leq \begin{cases}
1, & d=2k;\\
0, & d= 2k+1,
\end{cases} \quad
 \mathrm{Ih}^{d,0}(M) \leq \begin{cases}
1, & d=2k;\\
0, & d= 2k+1.
\end{cases}
\]
\end{lem}

\begin{proof}
We take a resolution $f: M' \to M$ with $M'$ nonsingular and projective. Since the decomposition theorem associated with $f$ is compatible with Hodge structures, we obtain that 
\[
\mathrm{Ih}^{0,d}(M) \leq h^{0,d}(M'), \quad \mathrm{Ih}^{d,0}(M) \leq h^{d,0}(M').
\]
The bounds of Lemma \ref{lem2.2} follow from \cite{GKKP, GKP1, KS}, that
\[
\oplus_{d\geq 0} H^0(M', \Omega^d_{M'}) = \oplus_{d\geq 0} H^0(M, \Omega^{[d]}_M)
\]
where the latter is generated by $\sigma$ by the definition of irreducible symplectic varieties.
\end{proof}

Now we consider the class of the reflexive symplectic form
\[
\sigma \in H^0(M, \Omega_M^{[2]}).
\]
By \cite[Theorem 8]{Sch}, the cohomology $H^2(M, \BC)$ carries a pure Hodge structure of weight 2 whose $(2,0)$-component is recovered by $H^0(M, \Omega_M^{[2]})$. Therefore, the class of the reflexive symplectic form $\sigma$ gives rise to a class $\sigma \in H^{2,0}(M) \subset H^2(M,\BC)$ which induces a nilpotent operator on the intersection cohomology
\[
\sigma\cup: \mathrm{IH}^d(M, \BC) \rightarrow \mathrm{IH}^{d+2}(M, \BC), \quad \mathrm{IH}^{*,*} \mapsto \mathrm{IH}^{*+2,*}.
\]
The Fujiki relations imply the vanishing $q_M(\sigma)=0$, therefore $\sigma$ is an isotropic class. We consider the increasing filtration (\ref{filtration}) induced by $\sigma$:
\[
 P^\sigma_0\mathrm{IH}^\ast(M, \BC) \subset P^\sigma_1\mathrm{IH}^\ast(M, \BC) \subset \dots \subset P^\sigma_k\mathrm{IH}^\ast(M, \BC) \subset \dots \subset \mathrm{IH}^\ast(M, \BC).
\]
In view of Proposition \ref{Prop1.4} and Lemma \ref{lem2.2}, Proposition \ref{prop22} follows from the two inequalities:
\begin{equation}\label{main1}
\dim P^\sigma_0\mathrm{IH}^d(M, \BC) \leq \mathrm{Ih}^{d,0}(M), \quad \dim \mathrm{Gr}^\sigma_d \mathrm{IH}^d(M, \BC) \leq \mathrm{Ih}^{0,d}(M).
\end{equation}

\subsection{Lefschetz pairs}
Before proving (\ref{main1}), we show in this section that the topology of the Lagrangian fibration $\pi: M \to B$ constrains the action of isotropic classes on the intersection cohomology $\mathrm{IH}^*(M, \BC)$. This will serve as a main ingredient of the proof of Proposition \ref{prop22}.

We say that a pair of isotropic classes 
\[
(\gamma, \gamma') \in H^2(M, \BC) \times H^2(M, \BC), \quad q_M(\gamma) = q_M(\gamma') = 0
\]
is a \emph{Lefschetz pair} if the following holds:
\begin{enumerate}
    \item[(i)] the cup product with $\gamma'$ satisfies
    \[
    \gamma'\cup: P^\gamma_i\mathrm{IH}^j(M, \BC) \to P^\gamma_{i+2}\mathrm{IH}^{j+2}(M, \BC);
    \]
    \item[(ii)] for any $k$ the cup product with $\gamma'^k$ induces an isomorphism on the graded pieces
    \[
     \gamma'^k \cup: \mathrm{Gr}^\gamma_{n-k} \mathrm{IH}^d(M, \BC) \xrightarrow{\simeq} \mathrm{Gr}^\gamma_{n+k} \mathrm{IH}^{d+2k}(M, \BC).
    \]
\end{enumerate}

Lefschetz pairs arise naturally in Lagrangian fibrations. For instance, for two isotropic classes $L$ and $\eta$ with $L$ the pullback of an ample class on the base $B$ and $\eta$ a $\pi$-relative ample class, the relative hard Lefschetz theorem implies that $(L, \eta)$ is a Lefschetz pair.

\begin{prop}\label{Prop2.3}
Any pair of isotropic classes $(\gamma, \gamma')$ with $(\gamma, \gamma')_M \neq 0$ forms a Lefschetz pair. Here $(-. -)_M$ denotes the symmetric bilinear form given by the BBF form $q_M(-)$.
\end{prop}

\begin{proof}
Recall the lattice $\Lambda$, the monodromy group $G_M$, and its subgroup $G^\circ_M$ from the proof of Proposition \ref{Prop1.4}. We consider the variety 
\[
D = \{(x,y)\in H^2(M, \BC)^{\times 2} \,|\, q_M(x)=q_M(y) = 0, \, (x, y)_M \neq 0\}
\]
where the group $\mathrm{SO}(\Lambda_\BC)$ acts diagonally. Since for a Lefschetz pair $(L, \eta)$ associated with a Lagrangian fibration $\pi: M \to B$ as above, all the pairs
\[
(\lambda L, \mu \eta), \quad \lambda, \mu \neq 0
\]
are Lefschetz, we obtain that any point in $D$ can be expressed as $(g\gamma, g\gamma')$ with $(\gamma, \gamma')$ a Lefschetz pair and $g \in \mathrm{SO}(\Lambda_\BC)$. On the other hand, the properties (i,ii) for Lefschetz pairs are expressed as Zariski closed conditions for points in the variety $D$, and are preserved under the monodromy group action. Therefore, the density of $G_M^\circ$ in $\mathrm{SO}(\Lambda_\BC)$ implies that any point in $D$ forms a Lefschetz pair. This completes the proof of the proposition.
\end{proof}

\subsection{Proof of Proposition \ref{prop22}}\label{Sec2.4}
In this section we show (\ref{main1}) and thus complete the proof of Proposition \ref{prop22}.

We first prove the second inequality of (\ref{main1}). By definition we have
\[
P^\gamma_{d-1}\mathrm{IH}^d(M, \BC) = \sum_{i\geq 1} \left(
\mathrm{Ker}(\sigma^{n+i-1}) \cap \mathrm{Im}(\sigma^{i-1}) \right) \cap \mathrm{IH}^d(M, \BC) \supset \mathrm{Ker}(\sigma^{n}).
\]
Hence we conclude that
\[
\dim\left(P^\gamma_{d}\mathrm{IH}^d(M, \BC)/P^\gamma_{d-1}\mathrm{IH}^d(M, \BC) \right) \leq \dim \left(\mathrm{IH}^d(M, \BC)/ \mathrm{Ker}(\sigma^n)\right) \leq  \mathrm{Ih}^{0,d}(M),
\]
which yields the desired inequality.

The following claim proves the the first inequality of (\ref{main1}).

\medskip
\noindent {\bf Claim.} $P^\sigma_0\mathrm{IH}^\ast(M, \BC) = \mathrm{IH}^{*,0}(M)$.

\begin{proof}[Proof of Claim]
Since the vector spaces $\mathrm{IH}^{*,0}(M)$ are spanned by powers of $\sigma$, we have immediately the inclusion 
\[
P^\sigma_0\mathrm{IH}^\ast(M, \BC) \supset \mathrm{IH}^{*,0}(M,\BC).
\]
Now assume that $P^\sigma_0\mathrm{IH}^\ast(M, \BC)$ is not contained in $\mathrm{IH}^{*,0}(M)$. Then we can find a nontrivial class $\gamma$ satisfying
\begin{equation}\label{gamma1}
\gamma \in P^\sigma_0\mathrm{IH}^d(M, \BC), \quad \gamma \in \bigoplus_{i>0} \mathrm{IH}^{d-i,i}(M, \BC).
\end{equation}
In particular, taking cup product with the complex conjugate $\overline{\sigma} \in H^{0,2}(M)$ of the reflexive symplectic form $\sigma$ together with the second equation of (\ref{gamma1}) yields the vanishing
\begin{equation}\label{contradiction}
\overline{\sigma}^n \gamma = 0 \in \mathrm{IH}^{2n+d}(M, \BC).
\end{equation}

On the other hand, $(\sigma, \overline{\sigma}) \in D$ forms a Lefschetz pair by Proposition \ref{Prop2.3}. Hence we deduce from the property (ii) of Lefschetz pairs that
\[
 \overline{\sigma}^n \cup: P^\sigma_0\mathrm{IH}^d(M, \BC)  \xrightarrow{\simeq} P^\sigma_{2n}\mathrm{IH}^{2n+d}(M, \BC)/P^\sigma_{2n-1}\mathrm{IH}^{2n+d}(M, \BC),
\]
which further implies that 
\[
\overline{\sigma}^n \gamma \neq 0
\]
for $\gamma$ as in (\ref{gamma1}). This contradicts (\ref{contradiction}), which completes the proof of the Claim.
\end{proof}

\section{Symplectic resolutions}

\subsection{LLV algebra for intersection cohomology}
We first recall the definition of the LLV algebra (extended to possibly singular varieties). 

Let $X$ be a projective variety of dimension $n$. An element $\alpha \in H^2(X, \BC)$ is called of \emph{Lefschetz type} if for any $k$, the cup product with $\alpha^k$ induces an isomorphism
\[\alpha^k \cup : \mathrm{IH}^{n - k}(X, \BC) \xrightarrow{\simeq} \mathrm{IH}^{n + k}(X, \BC).\]
In other words, the class $\alpha$ induces an $\mathfrak{sl}_2$-triple $(L_\alpha, H, \Lambda_\alpha)$ acting on $\mathrm{IH}^*(X, \BC)$. By the hard Lefschetz theorem for intersection cohomology, all ample classes are of Lefschetz type.

The LLV algebra of $X$, denoted by $\mathfrak{g}(X)$, is defined to be the Lie algebra generated by all~$\mathfrak{sl}_2$-triples associated with Lefschetz type classes. In the nonsingular case, the following result is due independently to Looijenga--Lunts and Verbitsky.

\begin{thm}[\cite{LL,Ver95,Ver96}] \label{thmllv}
For $M$ an irreducible symplectic manifold, there is a natural isomorphism of Lie algebras $\mathfrak{g}(M) \simeq \mathfrak{so}(b_2(M) + 2)$.
\end{thm}

The proof of Theorem \ref{thmllv} uses hyper-K\"ahler metrics and quaternions. Our goal is to extend this result to irreducible symplectic varieties admitting a symplectic resolution.

\subsection{Symplectic resolutions and BBF forms}

Let $M$ be an irreducible symplectic variety of dimension $2n$ admitting a symplectic resolution $f: M' \to M$. By \cite[Lemma 2.11]{Kal}, the map~$f$ is automatically semismall. Then, by the decomposition theorem for semismall maps~\cite{dCM00}, there is a canonical decomposition
\begin{equation} \label{eq:ss}
H^*(M', \BC) = \mathrm{IH}^*(M, \BC) \oplus V,
\end{equation}
where $V$ stands for the cohomology of those direct summands of $Rf_*\mathrm{IC}_{M'}$ supported on proper closed subsets of $M$. The identity \eqref{eq:ss} respects the pure Hodge structures on both sides. Moreover, the canonical inclusion
\begin{equation} \label{eq:inc}
\mathrm{IH}^*(M, \BC) \subset H^*(M', \BC)
\end{equation}
is a morphism of $H^*(M, \BC)$-modules via $f^*$.

We denote by 
\[q_M: H^2(M, \BC) \to \BC, \quad q_{M'}: H^2(M', \BC) \to \BC\]
the respective BBF form, normalized by $\sigma \in H^0(M, \Omega_M^{[2]})$ and $f^*\sigma \in H^0(M', \Omega_{M'}^2)$ such that
\[\int_M (\sigma\overline{\sigma})^n = \int_{M'} (f^*\sigma \overline{f^*\sigma})^n = 1.\]
By \cite[Lemma 23]{Sch}, there is the compatibility
\[q_M = q_{M'} \circ f^*.\]

\begin{lem} \label{lem:lef}
For $\alpha \in H^2(M, \BC)$, the following are equivalent:
\begin{enumerate}
    \item[(i)] $\alpha$ is of Lefschetz type on $M$;
    \item[(ii)] $f^*\alpha$ is of Lefschetz type on $M'$;
    \item[(iii)] $q_M(\alpha) = q_{M'}(f^*\alpha) \neq 0$.
\end{enumerate}
\end{lem}

\begin{proof}
The implication (i) $\Rightarrow$ (iii) follows from the Fujiki relations \cite[Theorem 2]{Sch}. As \eqref{eq:inc} is an inclusion of $H^*(M, \BC)$-modules, we also have (ii) $\Rightarrow$ (i). Finally, (iii) $\Rightarrow$ (ii) is \cite[Lemma~2.5]{SY}.
\end{proof}

\begin{rmk}
If $\alpha \in H^2(M, \BC)$ is an ample class, then the implication (i) $\Rightarrow$ (ii) also follows from the more general result \cite[Theorem 2.3.1]{dCM00} about lef line bundles.
\end{rmk}

\subsection{Proof of Theorems \ref{thm0.5} and \ref{thm0.4}}

Let $f: M' \to M$ be the symplectic resolution. We first prove Theorem \ref{thm0.4} by identifying the LLV algebra $\mathfrak{g}(M)$ with a Lie subalgebra of its counterpart $\mathfrak{g}(M')$.

Consider an element $\alpha \in H^2(M, \BC)$ of Lefschetz type, or equivalently, a Lefschetz $\mathfrak{sl}_2$-triple $(L_\alpha, H, \Lambda_\alpha)$ acting on $\mathrm{IH}^*(M, \BC)$. By Lemma \ref{lem:lef}, the class $f^*\alpha$ is also of Lefschetz type, hence a Lefschetz $\mathfrak{sl}_2$-triple $(L_{f^*\alpha}, H, \Lambda_{f^*\alpha})$ acting on $H^*(M', \BC)$. Since the inclusion~\eqref{eq:inc} is compatible with $H^2(M, \BC)$-actions, the restriction of $L_{f^*\alpha}$ to $\mathrm{IH}^*(M, \BC)$ is just $L_\alpha$. In other words, the matrix of $L_{f^*\alpha}$ under the decomposition \eqref{eq:ss} is block upper triangular. The same is true for $\Lambda_{f^*\alpha}$.

\begin{lem} \label{lem:linalg}
The restriction of $\Lambda_{f^*\alpha}$ to $\mathrm{IH}^*(M, \BC)$ is $\Lambda_\alpha$.
\end{lem}

\begin{proof}
Take $\gamma \in \mathrm{IH}^d(M, \BC)$. Since $\alpha$ is of Lefschetz type, there is the primitive decomposition
\[\gamma = \sum_{j}L_\alpha^j \gamma_j\]
with $\gamma_j \in \mathrm{IH}^{d - 2j}(M, \BC)$ satisfying $L_\alpha^{2n - d + 2j + 1}\gamma_j = 0$. Then, via the inclusion~\eqref{eq:inc}, we obtain the same primitive decomposition
\[\gamma = \sum_{j}L_{f^*\alpha}^j \gamma_j\]
this time viewed as in $H^d(M', \BC)$. By definition we have
\[\Lambda_{f^*\alpha}\gamma = \sum_{j}j(2n - d + j + 1)L_{f^*\alpha}^{j - 1}\gamma_j \in H^{d - 2}(M', \BC).\]
The lemma follows by comparing with
\[\Lambda_{\alpha}\gamma = \sum_{j}j(2n - d + j + 1)L_{\alpha}^{j - 1}\gamma_j \in \mathrm{IH}^{d - 2}(M, \BC). \qedhere\]
\end{proof}

Let $\mathfrak{g}' \subset \mathfrak{g}(M')$ denote the Lie subalgebra generated by all $\mathfrak{sl}_2$-triples associated with Lefschetz type classes in $f^*H^2(M, \BC) \subset H^2(M', \BC)$. The proof of Theorem \ref{thmllv} actually yields a natural isomorphism of Lie algebras $\mathfrak{g}' \simeq \mathfrak{so}(b_2(M) + 2)$; see \emph{e.g.}~\cite[Theorem 11.1]{Ver95}. On the other hand, the assignment
\[L_\alpha \mapsto L_{f^*\alpha}, \quad \Lambda_\alpha \mapsto \Lambda_{f^*\alpha}\]
together with Lemma \ref{lem:linalg} induces a surjective morphism of Lie algebras
\begin{equation} \label{eq:g'}
\mathfrak{g'} \to \mathfrak{g}(M).
\end{equation}

To prove that \eqref{eq:g'} is an isomorphism, we consider the subalgebra $\overline{H}^*(M, \BC) \subset H^*(M, \BC)$ generated by $H^2(M, \BC)$. Since $H^2(M, \BC)$ is pure, we also have inclusions
\[\overline{H}^*(M, \BC) \subset \mathrm{IH}^*(M, \BC) \subset H^*(M', \BC)\]
where the composition is induced by $f^*$. Let $\mathfrak{g}''$ denote the structure Lie algebra of $\overline{H}^*(M, \BC)$; see \cite[Section 8]{Ver95} for the terminology. By definition, the inclusion $\overline{H}^*(M, \BC) \subset \mathrm{IH}^*(M, \BC)$ induces a surjective morphism of Lie algebras
\[\mathfrak{g}(M) \to \mathfrak{g}''.\]
Moreover, the classification of reduced Lefschetz--Frobenius algebras \cite[Theorem 10.1]{Ver95} shows that $\mathfrak{g}'' \simeq \mathfrak{so}(b_2(M) + 2)$, and that the composition $\mathfrak{g}' \to \mathfrak{g}(M) \to \mathfrak{g}''$ is a natual isomorphism. This proves Theorem \ref{thm0.4}.

Once Theorem \ref{thm0.4} is established, the proof of the ``perverse = Hodge'' Theorem \ref{thm0.5} is identical to the nonsingular case \cite{SY}. It makes use of three key facts:
\begin{enumerate}
\item[(i)] cupping with the pullback of an ample class $L = \pi^*\beta$ controls the perverse filtration on $\mathrm{IH}^*(M, \BC)$; this is Proposition \ref{Prop1.3};

\item[(ii)] cupping with the reflexive symplectic form $\sigma$ (or its complex conjugate $\overline{\sigma}$) controls the Hodge filtration on $\mathrm{IH}^*(M, \BC)$; this follows from the analogous statement for the nonsingular $M'$ \cite[Theorem 1.4]{Ver90} together with the isomorphism of Hodge structures~\eqref{eq:ss};

\item[(iii)] one can deform the $\mathfrak{sl}_2$-triple associated with $L$ to the one associated with $\sigma$ inside the semisimple Lie algebra $\mathfrak{g}(M)$.
\end{enumerate}
We refer to \cite{SY} for the precise arguments. \qed

\begin{rmk}
{When $b_2(M)$ is at least 5, one can prove Theorem \ref{thm0.5} without any Lie algebra action. In fact, by Proposition \ref{Prop1.4}, it suffices to find two isotropic classes $\gamma_1,\gamma_2\in H^2(M,\BC)$ such that the associated filtrations $P_k^{\gamma_1},P_k^{\gamma_2}$ on $\mathrm{IH}^*(M,\BC)$ correspond respectively to the perverse and Hodge filtration. By Proposition \ref{Prop1.3}, one can take $\gamma_1=\pi^*\beta$ for some ample class $\beta\in H^2(B,\BC)$. For the Hodge filtration, one takes the complex conjugate of the reflexive symplectic form $\gamma_2=\overline{\sigma}\in H^2(M, \BC)$. The filtration $P_k^{f^*\overline{\sigma}}$ on $H^*(M', \BC)$ being the Hodge filtration, one sees that $P_k^{\overline{\sigma}}$ on $\mathrm{IH}^*(M, \BC)$ is the Hodge filtration via the isomorphism of Hodge structures \eqref{eq:ss}.}

\end{rmk}
\begin{rmk}
We expect both Theorems \ref{thm0.5} and \ref{thm0.4} to hold true for arbitrary irreducible symplectic varieties $M$. But proving them would amount to resolving two issues. One is to find a purely algebraic proof of the Looijenga--Lunts--Verbitsky Theorem \ref{thmllv} (without using hyper-K\"ahler metrics). The other is to show that the cup product with $\sigma$ controls the Hodge filtration on $\mathrm{IH}^*(M, \BC)$, which is a priori not clear.
\end{rmk}

\end{document}